\documentclass[11pt]{article}

\usepackage[a4paper]{anysize}\marginsize{3.5cm}{3.5cm}{1.3cm}{2.5cm}
\pdfpagewidth=\paperwidth \pdfpageheight=\paperheight 

\usepackage[latin1]{inputenc}
\usepackage{amssymb}

\usepackage{latexsym}
\usepackage{amsfonts}
\usepackage{amsmath}
\usepackage{amscd}

\usepackage{tikz}

\makeatletter\renewcommand\@seccntformat[1]{\csname the#1\endcsname.\enspace}\makeatother

\newtheorem{introtheorem}{Theorem}  

\newtheorem{thm}{Theorem}[section]
\newtheorem{theorem}[thm]{Theorem}

\newtheorem{proposition}[thm]{Proposition}
\newtheorem{corollary}[thm]{Corollary}

\newcommand\mkthm[2]{\newenvironment{#1}[1][]{\def\thmarg{##1}\ifx\empty\thmarg\begin{#2}\else\begin{#2}[##1]\fi\rm}{\end{#2}}}
 \mkthm{definition}{tdefinition}
         \mkthm{remark}{tremark}
       \mkthm{example}{texample}
\newtheorem{tquestion}[thm]{Question}     \mkthm{question}{tquestion}

\newtheorem{thevarthm}[thm]{\varthmname}

\newenvironment{varthm*}[1]{\trivlist\item[]{\bf #1.}\it}{\endtrivlist}

\newenvironment{proof}[1][Proof]{\trivlist\item[\hskip\labelsep{\textit{#1.}}]}{\hspace*{\fill}$\Box$\endtrivlist}

\let\tilde=\widetilde

\renewcommand\ge{\geqslant}  
\renewcommand\le{\leqslant}  

\newcommand\grant[1]{{\renewcommand\thefootnote{}\footnotetext{#1.}}}
\newcommand\keywords[1]{{\renewcommand\thefootnote{}\footnotetext{\textit{Keywords:} #1.}}}
\newcommand\subclass[1]{{\renewcommand\thefootnote{}\footnotetext{\textit{Mathematics Subject Classification (2010):} #1.}}}

\newcommand\N{\mathbb N}
\newcommand\Zet{\mathbb{Z}}

\newcommand\be{\begin{eqnarray*}}
\newcommand\ee{\end{eqnarray*}}

\newcommand\rr{\rightskip0pt plus 1fil\relax}
\newcommand\tline{\noalign{\vskip0.4ex}\hline\noalign{\vskip0.65ex}}
\newcommand\midtext[1]{\quad\mbox{#1}\quad}
\newcommand\compact{\itemsep=0cm \parskip=0cm}
\newcommand\set[1]{\left\{#1\right\}}
\newcommand\with{\,\,\vrule\,\,}

\newcommand\eps{\varepsilon}
\newcommand\newop[2]{\newcommand#1{\mathop{\rm #2}\nolimits}}

\newop\mult{mult}
\newop\gon{gon}
\newop\Num{Num}

\newcommand\mmax{m_{\rm max}}
\newcommand\dmin{d_{\rm min}}
\newcommand\roundup[1]{\left\lceil#1\right\rceil}
\newcommand\rounddown[1]{\left\lfloor#1\right\rfloor}

\begin{document}

\title{Seshadri constants on abelian and bielliptic surfaces -- potential values and lower bounds}
\author{\normalsize Thomas Bauer and {\L}ucja Farnik}
\date{\normalsize August 17, 2020}
\maketitle
\grant{The second author was partially supported by National Science Centre, Poland, grant 2018 /28/C/ST1/00339}
\keywords{Seshadri constants, abelian surfaces, bielliptic surfaces}
\subclass{14C20, 14K05}


\begin{abstract}
   In this note we contribute to the study of Seshadri constants on abelian and bielliptic surfaces.
   We specifically focus on
   bounds
   that hold
   on all such surfaces,
   depending only on the self-intersection
   of the ample line bundle under consideration.
   Our result improves previous bounds and it
   provides rational numbers as bounds, which are potential Seshadri constants.
\end{abstract}


\section{Introduction}

   The purpose of this note is to contribute to ongoing efforts in
   bounding Seshadri constants of ample line bundles
   on smooth surfaces,
   and to provide restrictions on their
   possible submaximal values.

   Recall that
   for an ample line bundle $L$ on a smooth projective surface
   $X$, the \emph{Seshadri constant} $\eps(L,x)$ at a point
   $x\in X$ is by definition the real number
   \be
      \varepsilon(L,x)=\inf \left\{\frac{L\cdot C}{\mult_x C} \with C\subset X \mbox{ irreducible curve through $x$}\right\}
   \ee
   (see \cite{Ba-et-al2009} for more about the background on Seshadri
   constants, and for their basic properties.)
   Naturally, one of the important problems in this area of research
   consists in bounding or even computing
   Seshadri constants.
   It was recognized early on that
   bounding the multiplicities $m=\mult_xC$ of
   irreducible curves $C\subset X$ in terms of their
   self-intersection $C^2$ can be an
   effective means in order to obtain lower bounds on $\eps(L,x)$.
   The first result in this direction is due to
   Ein and Lazarsfeld
   \cite{Ein-Lazarsfeld:sesh}, who showed that
   $C^2\ge m(m-1)$ holds, if $C$ moves in a family of curves
   $(C_t)$ with
   multiplicities $\mult_x C_t\ge m$.
   Under suitable assumptions, Xu \cite{Xu:ample-line-bundles}
   improved the bound to
   $C^2\ge m(m-1)+1$.
   Knutsen, Syzdek, and Szemberg~\cite{KSSz2009}
   and, independently, Bastianelli~\cite{Bas2009}
   provided a further improvement
   by showing that if
   $C$ moves in a
   2-dimensional family of curves with multiplicity at least $m$, then
   $C^2\ge m(m-1)+\gon(\tilde C)$, where $\gon(\tilde C)$ is the
   gonality of the
   normalization of $C$.
   As these results work under the assumption that the curves
   move in families, they
   do not lead to bounds on Seshadri constants at arbitrary points,
   but at very general points (as in \cite{Ein-Lazarsfeld:sesh})
   or outside of a finite number of curves
   (as in \cite{Xu:ample-line-bundles}).

   In the present note, we are interested in bounds of this type,
   which however apply to arbitrary curves, and therefore lead to bounds
   on Seshadri constants at arbitrary points.
   On abelian surfaces, it is well-known that one has
   $C^2\ge m(m-1)+2$ for \emph{all} non-elliptic curves, and
   we show that the same bound also holds on
   bielliptic surfaces
   (Proposition~\ref{prop:bound-bielliptic-irred}).
   We use these bounds to obtain information about
   the possible rational numbers that might occur as Seshadri constants,
   and to find the smallest rational values
   in these sets.
   We show:

\begin{introtheorem}\label{thm:intro}
   Let $X$ be an abelian surface
   or a bielliptic surface,
   let $L$ be an ample line
   bundle on $X$, and let $x\in X$ be any point.
   Suppose that $\eps(L,x)<\sqrt{L^2}$ and that
   $\eps(L,x)$ is not computed by an elliptic curve (if $X$ is abelian)
   resp.\ that it is not computed by a fiber (if $X$ is bielliptic).

   Then $\eps(L,x)$ is one of the rational numbers in the set
   \be
      \set{\frac dm\with d^2 \ge L^2(2+m(m-1)), \ m\ge 2}
      \,,
   \ee
   and for any $L^2\ge 2$ we have the lower bound
   \be
      \eps(L,x)\ge\min\left\{
         \frac{\roundup{\sqrt{4L^2}}}{2},
         \frac{\roundup{\sqrt{8L^2}}}{3},
         \frac{\roundup{\sqrt{14L^2}}}{4},
         \frac{\roundup{\sqrt{22L^2}}}{5},
         \frac{\roundup{\sqrt{32L^2}}}{6},
         \frac{\roundup{\sqrt{44L^2}}}{7}
      \right\}.
   \ee
   Moreover, if $L^2\ge 4982$, then
   $$
      \varepsilon(L,x)\ge \frac{\roundup{\sqrt{14L^2}}}4.
   $$
\end{introtheorem}

   Our interest in Theorem~\ref{thm:intro} lies in the fact that it
   improves
   previous results in two respects. First, it
   is closer to the
   general upper bound $\sqrt{L^2}$ than the previous bounds,
   and it provides rational numbers
   $d/m$ as estimates that represent potential Seshadri constants,
   while the previous irrational bounds are theoretical by design.
   (We provide a more detailed
   comparison in Section~\ref{Comparison_with_previous}.)
   And secondly, on bielliptic surfaces the bound not only applies
   to very general points, but
   to arbitrary points
   (see Remark~\ref{rem:not_only_v_gen}).
   Note that one cannot hope for
   bounds expressed in simple formulas that are at the same
   time sharp: The case of abelian surfaces
   \cite[Section~6]{Bauer:sesh-alg-sf}
   shows that, even in the case of
   Picard number one, the actual values
   of $\eps(L,x)$ are not given by simple algebraic expressions.

   Beyond abelian and bielliptic surfaces,
   our method of proof for Theorem~\ref{thm:intro} works
   more generally on surfaces satisfying the following property:
   \begin{itemize}
   \item[$(\star)$]
      For any
      an irreducible curve $C\subset X$, if  $C^2 > 0$ and $m=\mult_xC$, then $C^2 \ge m(m-1)+2$.
   \end{itemize}
   In fact, we obtain Theorem~\ref{thm:intro} as a consequence of
   a result in this more general setting
   (see Theorem~\ref{thm:bound_for_smooth}).
   For the argument to work,
   condition~$(\star)$ need not hold for all
   irreducible curves, but only for those that are
   submaximal for some ample line bundle.
   It would be interesting to explore further,
   whether this can be used to obtain bounds on other kinds of surfaces.


\section{A bound on Seshadri constants on smooth surfaces}

   Our aim in this section is
   to find the smallest rational values
   which could be Seshadri constants of ample line bundles
   on smooth projective surfaces satisfying
   the property $(\star)$ that was stated in the introduction.
   Our main result is Theorem~\ref{thm:bound_for_smooth}.
   In finding potential rational values of Seshadri constants we were
   inspired by a result
   of Szemberg \cite{Sz2012:Bounds-on-Sesh}
   for smooth projective surfaces with Picard number 1,
   whose method of proof however
   does not extend readily to higher Picard number.

   Note first that according to
   \cite[Proposition~2.1]{Bauer-Szemberg:Seshadri-adjoint},
   every positive rational number occurs
   as the Seshadri constant $\eps(L,x)$ for some ample line $L$
   bundle on some smooth surface at some point $x$.
   By contrast, we point out that
   for a fixed line bundle $L$, the possibilities
   are limited:

\begin{proposition}\label{prop:set}
   Let $X$ be a smooth projective surface satisfying property $(\star)$. Let $L$ be an ample line
   bundle on $X$ and $x\in X$. Suppose that $\eps(L,x)<\sqrt{L^2}$
   and that $\eps(L,x)$ is computed by a  curve with $C^2 > 0$. Then $\eps(L,x)$ is
   one of the rational numbers
   in the following set
   \be
      \set{\frac dm\with d^2 \ge L^2 (2+m(m-1)), \ m\ge 2}.
   \ee
\end{proposition}

\begin{proof}
   Suppose that $m=1$.
   Then by the assumption on $\eps(L,x)$ we have $\frac{LC}1<\sqrt{L^2}$. Using the Hodge Index Theorem we obtain that $C^2\le 0$, which contradicts
   the assumption on $C$.

   Suppose then $m>1$.
    The Seshadri constant $\eps(L,x)$ is computed as  $\eps(L,x)=\frac dm$ with $d=L\cdot C$ and $m=\mult_x C$.
   Furthermore, using the Hodge Index Theorem, the assumption on $C$, and property $(\star)$,
   we get
   \be
      d^2 = (L\cdot C)^2 \ge L^2 C^2 \ge L^2 (2 + m(m-1))
      \,,
   \ee
   as claimed.
\end{proof}

   In the setting of Proposition~\ref{prop:set}
   let $N:=L^2$, and consider the set
   \be
      \Omega=\set{(d,m)\in\N^2\with d^2 \ge N(2+m(m-1)),\ m\ge 2}.
   \ee
   Similarly to
   \cite{Sz2012:Bounds-on-Sesh},
   by Proposition~\ref{prop:set}
   the issue in the problem of bounding Seshadri constants
   becomes to minimize the ratio $d/m$ of elements $(d,m)\in\Omega$.
   For fixed $d$, the maximal $m$ such that $(d,m)$ lies in
   $\Omega$ is given by
   \be
      \mmax(d) = \rounddown{ \frac12 + \sqrt{\frac{d^2}{N}-\frac74} }.
   \ee
   And for fixed $m$, the minimal $d$ such that $(d,m)$ lies in
   $\Omega$ is given by
   \be
      \dmin(m) = \roundup{\sqrt{N(2+m(m-1))}}.
   \ee
   Therefore, if we know that $\eps(L,x)$ is
   computed by a curve of degree $d$, then
   \be
      \eps(L,x)\ge \frac d{\mmax(d)}=\frac d{\rounddown{ \frac12 + \sqrt{\frac{d^2}{N}-\frac74} }}
   \ee
   and if we know that $\eps(L,x)$ is computed by a curve of
   multiplicity $m$, then
   \be
      \eps(L,x)\ge \frac{\dmin(m)}m=\frac{\roundup{\sqrt{N(2+m(m-1))}}}{m}.
   \ee
   The latter inequality bounds $\eps(L,x)$,
   but it does so in a rather ineffective way, since
   there are infinitely many possible values of the unknown $m$.
   Our result shows that only finitely many of them
   need to be taken into account:

\begin{theorem}\label{thm:bound_for_smooth}
   Let $X$ be a smooth projective surface satisfying property $(\star)$. Let $L$ be an ample line
   bundle on $X$ and $x\in X$. Suppose that $\eps(L,x)<\sqrt{N}$
   and that $\eps(L,x)$ is computed by a  curve with $C^2 > 0$.
   Then for any $N\ge 2$, we have
   $$\varepsilon(L,x)\ge \min\left\{\frac{\dmin(m)}m \with m\in\{2,\ldots,7\}\right\}\,.$$
   Moreover for $N$ big enough, we have
   $$\varepsilon(L,x)\ge \frac{\dmin(4)}4\,.$$
\end{theorem}

\begin{proof}
   Consider the two functions $f$ and $g$ defined by
   \be
      f(N,m)=\frac{\roundup{\sqrt{N(2+m(m-1))}}}{m}
      \midtext{and}
      g(N,m)=\frac{{\sqrt{N(2+m(m-1))}}}{m}
      \,.
   \ee
   We will prove the first part of the theorem by showing that for a fixed $N$ and for $m\ge 8$
   \begin{equation}\label{f(N,m)>f(N,7)}
   f(N,m) \ge f(N,7).
   \end{equation}
   We start with proving a stronger inequality for large $N$:
   We decrease the left-hand side and increase the right-hand side, and we wish to prove that for $m\ge 8$
   $$g(N,m)\ge g(N,7)+\frac 1 7\,.$$

   Computing the derivative of $g(N,m)$ with respect to $m$ we obtain that $g(N,m)$ is an increasing function of $m$ in the interval $(4,\infty)$.
   So it is enough to prove that $$g(N,8)\ge g(N,7)+\frac 1 7\,.$$
   It can be easily computed that the inequality
   $\frac{{\sqrt{58N}}}{8} \ge\frac{{\sqrt{44N}}}{7} +\frac 1 7$ holds for $N\ge 1072$. Therefore for $m\ge 8$ and a fixed $N\ge 1072$ we have
   $$f(N,m)\ge g(N,m)\ge g(N,8)\ge g(N,7)+\frac 1 7\ge f(N,7)\,.$$
   For the remaining finitely many cases, i.e., for $N\in [2, 1070]$, we check with Maple software that the original inequality (\ref{f(N,m)>f(N,7)}) is satisfied.
   This proves that  $$\varepsilon(L,x)\ge \min\left\{\frac{f(N,m)}m \with m\in\{2,\ldots,7\}\right\}\,.$$

   Now we will prove that for $N$ big enough, $\min\left\{\frac{\dmin(m)}m \with m\in\{2,\ldots,7\}\right\}=\frac{\dmin(4)}4$.
   Analogously, it remains to check whether for any $m$ (in fact $m\in \{2,\ldots,7\}$ is enough) and for $N$ big enough  $$g(N,m)\ge g(N,4)+\frac 1 4\,.$$
   Equivalently, we ask if the following inequality holds for large $N$:
   \begin{equation}\label{d_min_for_big_N}
    \frac{{\sqrt{m(m-1)+2}}}{m} -\frac{{\sqrt{14}}}{4} \ge\frac 1 {4\sqrt{N}}
    \,.
   \end{equation}
   It can be confirmed by a computation that for any $m\neq 4$, the left-hand side is a positive number. This completes the proof.
\end{proof}

\begin{remark}\label{rmk:maple}
   If we consider equation (\ref{d_min_for_big_N}) for  all
   $m\in\{2,3,5,6,7\}$, we obtain that
   $$
      \min\left\{\frac{\dmin(m)}m \with m\in\{2,\ldots,7\}\right\}=\frac{\dmin(4)}4 \textrm{ for all } N\ge 8776.
   $$
   Checking the original formula
   (\ref{f(N,m)>f(N,7)}) involving the round-up for the
   remaining finitely many values of $N$ using Maple software reveals
   that in fact
   $$
      \frac{\dmin(4)}4=\min\left\{\frac{\dmin(m)}m \with m\in\{2,\ldots,7\}\right\} \ \textrm{for all}\ N\ge 4982
      \,,
   $$
   and this is the bound on $N$ stated in Theorem~\ref{thm:intro}.
\end{remark}

\begin{remark}
   Theorem~\ref{thm:bound_for_smooth} and Remark~\ref{rmk:maple}
   along with Maple computations show that the minimum of ratios
   $\min\left\{\frac{\dmin(m)}m \with m\in\{2,\ldots,7\}\right\}$ is attained at
   \begin{itemize}\compact
   \item $m=2$,  1 time (for $N=4$),
   \item $m=3$,  59 times (for certain $N\le 1012$),
   \item $m=5$,  274 times (for certain $N\le 4980$),
   \item $m=6$,  9 times (for certain $N\le 294$),
   \item $m=7$,  1 time (for $N=42$).
   \end{itemize}
   In all other cases the minimum is attained at $m=4$.
\end{remark}

\begin{tquestion}
   Find a formula
   in terms of $N\in [2, 4980]$
   that expresses the value of $m$, at which the minimum
   of the ratios $\frac{d_{\rm min}(m)}m$ is attained.
\end{tquestion}

   In view of \cite{Sz2012:Bounds-on-Sesh}
   it is not clear whether a simple formula can be expected
   as an answer to this question.


\section{Application to abelian surfaces and bielliptic surfaces}\label{sect:abelian-bielliptic}

   Our aim is now to derive Theorem~\ref{thm:intro} from
   Theorem~\ref{thm:bound_for_smooth}.
   The following bound on the self-intersection of irreducible
   curves on abelian surfaces is well-known.

\begin{proposition}\label{prop:bound-abelian-irred}
   Let $C$ be
   a non-elliptic irreducible curve $C$
   on an abelian surface $X$.
   Then
   \be
      C^2 \ge 2 + \sum_i m_i (m_i-1)
   \ee
   where the sum runs over all singularities of $C$ and $m_i$ are
   their respective multipli\-cities.
\end{proposition}

   The proposition follows from the fact that on abelian surfaces
   there are no rational curves, and all curves of geometric genus
   1 are smooth.
   Since on abelian surfaces there are no negative curves and
   the only curves with self-intersection 0 are elliptic curves,
   Theorem \ref{thm:bound_for_smooth}
   clearly implies the statement of
   Theorem~\ref{thm:intro} for the case of abelian surfaces:

\begin{corollary}\label{cor:bound_for_Seshadri_on_abelian}
   Let $X$ be an abelian surface and let $L$ be an ample line
   bundle on $X$. Let $N:=L^2$. Suppose that $\eps(L,x)<\sqrt{N}$ and that
   $\eps(L,x)$ is computed by a non-elliptic curve.
   Then for any $N\ge 2$
   $$\varepsilon(L,x)\ge \min\left\{\frac{\dmin(m)}m \with m\in\{2,\ldots,7\}\right\}\,.$$
   Moreover for $N\ge 4982$
   $$\varepsilon(L,x)\ge \frac{\dmin(4)}4\,.$$
\end{corollary}

Note that in the remaining case, where $\eps(L,x)$ is computed by an
elliptic curve, the possible values of $\eps(L,x)$ are clear: they are
the integers from 1 to $\rounddown{\sqrt{L^2}}$.

   In order to
   apply Theorem~\ref{thm:intro} to bielliptic surfaces,
   we will use a version of Proposition
   \ref{prop:bound-abelian-irred}
   for reducible curves on
   abelian surfaces, which we prove now.

\begin{proposition}\label{prop:bound-abelian-red}
   Let $C$ be a reduced (but possibly reducible)
   curve on an abelian
   surface. Suppose that $C$ has $r$ components, none
   of which is an elliptic curve.
   Then
   \be
      C^2 \ge 2r + \sum_i m_i (m_i-1).
   \ee
\end{proposition}

\begin{proof}
   We will argue by induction on $r$.
   The assertion is true by the previous proposition when $r=1$.
   So assume that $r\ge 2$ and decompose $C$
   in any way as a sum of curves $C=A+B$.
   By induction, the assertion is true for
   $A$ and for $B$. So, denoting by $s$ and $t$ the number of
   irreducible components of $A$ and $B$, we know that
   \be
      A^2 \ge 2s + \sum_i a_i (a_i-1) \midtext{and}
      B^2 \ge 2t + \sum_i b_i (b_i-1)
      \,,
   \ee
   where $a_i$ and $b_i$ are the multiplicities of $A$ resp.\ $B$
   at the singularities of $C$.
   So
   \be
      C^2 & = & A^2 + B^2 + 2A\cdot B \\
          & \ge & 2s + \sum_i a_i (a_i-1) + 2t + \sum_i b_i (b_i-1) + 2 \sum_i a_i b_i
      \,,
   \ee
   where the last term comes from the intersection inequality
   $A\cdot B\ge \sum_i a_ib_i$.
   Collecting terms
   we get
   \be
      C^2 \ge 2(s+t) + \sum_i (a_i+b_i)(a_i+b_i-1)
   \ee
   and using $m_i=a_i+b_i$ as well as $r=s+t$ this gives the
   assertion.
\end{proof}

   This version
   allows us to obtain an analogue of Proposition~\ref{prop:bound-abelian-irred}
   for bielliptic surfaces.

\begin{proposition}\label{prop:bound-bielliptic-irred}
   Let $C$ be
   an irreducible curve $C$
   on a bielliptic surface $X$
   that is not an elliptic curve.
   Then
   \be
      C^2 \ge 2 + \sum_i m_i (m_i-1)
      \,,
   \ee
   where the sum runs over all singularities of $C$ and $m_i$ are
   their respective multipli\-cities.
\end{proposition}

\begin{proof}
   The surface $X$ is the
   image of an abelian surface $Y$ under an
   unramified map $f:Y\to X$ (see \cite{BM1990:automorph-group}).
   Let $e=\deg f$.
   So every point of multiplicity $m$ on $C$
   gives rise to $e$ points of the same multiplicity $m$ on the
   pull-back $f^*C$.
   None of the components of $f^*C$ can be an elliptic curve,
   so we can apply
   Proposition~\ref{prop:bound-abelian-red} to obtain
   a bound on the
   self-intersection of the pull-back $f^*C$,
   \be
      (f^*C)^2 \ge 2s + e \sum_i m_i (m_i-1)
      \,,
   \ee
   where $s$ is the number of components of $f^*C$.
   Thus we get
   \be
      C^2 \ge \frac{2s} e + \sum_i m_i (m_i-1)
      \,.
   \ee
   So we have established in
   particular that $C^2$ is at least
   the
   sum on the right-hand side.
   The crucial point is now that this sum is an even number. As
   the
   intersection form on bielliptic surfaces is
   even, this implies that $C^2$ must differ from the sum by at least
   2, and this gives the assertion.
\end{proof}

   Therefore we have shown that Property $(\star)$ holds on
   bielliptic surfaces. If on a bielliptic surface  $\eps(L,x)$
   is computed by a curve $C$ different from a fibre, then
   we have
   $C^2 > 0$ (see Remark \ref{rmk:line-bundles-bielliptic} in the appendix).
   Hence
   Theorem~\ref{thm:bound_for_smooth}
   yields the following statement
   for bielliptic surfaces.

\begin{corollary}\label{cor:bound_for_Seshadri_on_bielliptic}
   Let $X$ be a bielliptic surface and let $L$ be an ample line
   bundle on $X$. Let $N:=L^2$. Suppose that $\eps(L,x)<\sqrt{N}$
   and that $\eps(L,x)$ is not computed by a fibre.
   Then for any $N\ge 2$
   $$\varepsilon(L,x)\ge \min\left\{\frac{\dmin(m)}m \with m\in\{2,\ldots,7\}\right\}\,.$$
   Moreover for $N\ge 4982$
   $$\varepsilon(L,x)\ge \frac{\dmin(4)}4\,.$$
\end{corollary}

The remaining case of $\eps(L,x)$ computed by a fibre is analogous to the case of  $\eps(L,x)$ computed by an elliptic curve on an abelian surface.


\section{Comparison with previously known results}\label{Comparison_with_previous}

   For abelian surfaces and bielliptic surfaces
   several lower bounds on Seshadri constants
   of a similar flavor are available in the literature.
   It is therefore interesting to see how exactly they
   compare with
   each other and with
   the bound given in
   Theorem~\ref{thm:intro}. We provide such a comparison in this section.
   Also, we show how
   Proposition~\ref{prop:bound-bielliptic-irred} can be used to
   generalize
   results of
   Hanumanthu and Roy
   (see Remark~\ref{rem:not_only_v_gen}) on bielliptic surfaces.

   We start with a result
   by Syzdek and Szemberg
   \cite{SSz2010}, which applies to any smooth projective surface:

\begin{theorem}[{\cite[Corollary 3.3]{SSz2010}}]
   Let $X$ be a smooth projective surface and let $L$ be an
   ample line bundle on $X$. Then
   $$
      \varepsilon(L, x) \ge \sqrt \frac 7 9\sqrt{L^2}
      \,,
   $$
   for very general $x\in X$,
   or $X$ is fibred by Seshadri curves,
   or $X$ is a cubic surface in $\mathbb{P}^3$ and $L =
   \mathcal{O}_X(1)$.
\end{theorem}

   For abelian surfaces the following bound was shown by the first author and
   Szemberg:

\begin{theorem}[{\cite[Theorem A.1]{Bauer:sesh-and-periods}}]
   Let $X$ be an abelian surface, $L$ an ample line,
   and $x\in X$ any point.
   If the Seshadri constant of $L$ at $x$ is
   computed by a non-elliptic curve, then $$\varepsilon(L, x)
   \ge \sqrt \frac 7 8\sqrt{L^2}\,.$$
\end{theorem}

   A number of results for Seshadri constants on bielliptic surfaces at very general points were obtained
   by Hanumanthu and Roy
   in \cite{HR2017}. We cite two of their results:

\begin{theorem}[{\cite[Theorem 3.9]{HR2017}}]\label{thm:0.93sqrtL^2}
   Let $X$ be a bielliptic surface and let $L$ be an ample line bundle on $X$. Suppose that $C\equiv(\alpha,\beta)$ is an irreducible, reduced curve with
   with $\alpha\neq 0$, $\beta\neq 0$, passing through a very general point with multiplicity $m\ge 1$.
   Then $$\frac{L\cdot C}m\ge (0.93)\sqrt{L^2}\,.$$
\end{theorem}

   As a corollary to Theorem \ref{thm:0.93sqrtL^2}, the following result was obtained in \cite{HR2017}.

\begin{theorem}[{\cite[Theorem 3.11]{HR2017}}]
   Let $X$ be a bielliptic surface and let $L$ be an ample line
   bundle on $X\cong  (E×F)/G$. If
   for a very general point $x\in X$ one has
   $\eps(L,x)< (0.93)\sqrt{L^2}$
   then $\eps(L,x)=\min\{L\cdot E, L\cdot F\}$.
\end{theorem}

\begin{remark}\label{rem:not_only_v_gen}
   We can use
   Proposition \ref{prop:bound-bielliptic-irred}
   to show that Theorem \ref{thm:0.93sqrtL^2} not only holds for very
   general points on a bielliptic surface, but for any point $x$
   where the Seshadri constant $\eps(L,x)$ is not computed by a
   fibre.
   Indeed, in the proof in~\cite{HR2017} the authors
   use the inequality in
   Property $(\star)$
   for very general points,
   with reference to
   the Xu-type lemma $C^2 \ge \left(\sum_{i=1}^r
   m_i^2\right)-m_1+\text{gon}(\widetilde{C})$ (see
   \cite[Lemma 2.2]{Bas2009} or \cite[Theorem A]{KSSz2009})
   and the fact that any
   bielliptic surface $X$ is nonrational, so that for every curve
   $C\subset S$ one has $\text{gon}(\widetilde{C})\ge 2$.
   Proposition \ref{prop:bound-bielliptic-irred}
   now tells us that the inequality in
   Property $(\star)$ holds for \emph{all} points,
   and hence the theorem generalizes in this respect.
\end{remark}

   Let us now compare
   the various bounds with each other and with our
   bound from Theorem \ref{thm:bound_for_smooth}.
   We have
   \be
      \frac{\dmin(m)}m&=&\frac{\roundup{\sqrt{N(2+m(m-1))}}}{m}\ge \frac{{\sqrt{N(2+m(m-1))}}}{m} \\
      &\ge& \frac{{\sqrt{14N}}}{4}=\sqrt{\frac{7}8N}\ge 0.93 \sqrt{N} > \sqrt\frac 7 9\sqrt{N}.
   \ee
   This shows in particular
   the bounds given by Corollaries \ref{cor:bound_for_Seshadri_on_abelian} and  \ref{cor:bound_for_Seshadri_on_bielliptic} improve the bounds given
   in
   \cite[Corollary 3.3]{SSz2010},
   \cite{Bauer:sesh-and-periods}, and \cite[Theorem 3.9]{HR2017}.
   To convey some feeling for the actual numbers, we
   present a table which shows the bounds in some chosen cases
   (Table~\ref{tab:comparison_of_bounds}).
   \begin{table}[h]
      \begin{tabular}{r|p{3.6cm}|p{3.6cm}|p{3.8cm}}
         $L^2$ & \rr Bound for abelian surfaces from \cite{Bauer:sesh-and-periods} & \rr Bound for bielliptic surfaces from \cite{HR2017} & \rr New bound for surfaces satisfying $(\star)$ \\ \tline
         2     & 1,3229                                                            & 1,3152                                               & 1,3333\\
         6     & 2,2913                                                            & 2,2780                                               & 2,3333\\
         8     & 2,6458                                                            & 2,6304                                               & 2,6667\\
         10    & 2,9580                                                            & 2,9409                                               & 3\\
         50    & 6,6144                                                            & 6,5761                                               & 6,6667\\
         100   & 9,3541                                                            & 9,3                                                  & 9,4\\
         5000  & 66,1439                                                           & 65,7609                                              & 66,25\\
         20000 & 132,2676                                                          & 131,5219                                             & 132,5\\
      \end{tabular}
      \caption{Table of bounds for $\eps(L,x)$.}\label{tab:comparison_of_bounds}
   \end{table}
   Apart from the small but noticable numerical improvement,
   one could argue that
   the most interesting feature of the new bound
   is the fact that it is always a rational number
   $d/m$ that represents a potential Seshadri constant,
   while the previous bounds are
   irrational numbers, which are therefore
   theoretical by design.

\section{Appendix on bielliptic surfaces}

   In this section we provide background on bielliptic surfaces.
   Remark \ref{rmk:line-bundles-bielliptic} was used in Section~\ref{sect:abelian-bielliptic}.

\begin{definition}
   A bielliptic surface $X$ (also called hyperelliptic)
   is a surface with Kodaira  dimension equal to $0$ and  irregularity  $q(S)=1$.
\end{definition}

   The canonical divisor $K_X$ on any bielliptic surface is numerically trivial, but non-zero.

   Alternatively (see \cite[Definition VI.19]{Bea1996}), a surface $X$ is bielliptic if  $X\cong (E\times F)/G$, where $E$ and $F$ are elliptic curves, and $G$ is an abelian group acting on $E$ by translations and acting on $F$, such that  $E/G$ is an elliptic curve and $F/G\cong \mathbb{P}^1$.
   Hence we have the following situation
\begin{center}
      \begin{tikzpicture}[line cap=round,line join=round,x=1.0cm,y=1.0cm, scale=0.5]
         \draw [->,line width=0.4pt] (7.,4.) -- (5.,2.5);
         \draw [->,line width=0.4pt] (8.,4.) -- (10.,2.5);
         \draw[color=black] (7.7,4.5) node {$S\cong (E\times F)/G$};
         \draw[color=black] (4.7,1.9) node {$E/G$};
         \draw[color=black] (10.8,1.9) node {$F/G\cong \mathbb{P}^1$};
         \begin{scriptsize}
            \draw[color=black] (5.6,3.5) node {$\Phi$};
            \draw[color=black] (9.4,3.5) node {$\Psi$};
         \end{scriptsize}
      \end{tikzpicture}
\end{center}
   where $\Phi$ and $\Psi$ are the natural projections.

   There are seven non-isomorphic groups that can act on
   $E\times F$. Two of them act on any $E\times F$, the other
   five require $F$ to be an elliptic curve of a specific form
   (see Table \ref{tab:action}).

   Following C. Bennett and R. Miranda
   \cite{BM1990:automorph-group}, let us fix the notation. Let
   $E=\mathbb{C}/(\mathbb{Z}\tau_1+\mathbb{Z})$ and
   $F=\mathbb{C}/(\mathbb{Z}\tau_2+\mathbb{Z})$, where
   $\tau_1,\tau_2\in \mathbb{C}$. Let $\zeta$ be the sixth root
   of unity, i.e., $\zeta=e^{\pi i/3}$.

\begin{proposition}[{\cite[Table 1]{BM1990:automorph-group}}, see also {\cite[VI.20]{Bea1996}}]\label{prop:action}
   The seven types of bielliptic surfaces are described in Table \ref{tab:action}.
   \begin{table}[h]
      $$
      \begin{array}{c|c|l|l}
         \textrm{Type}
         &\tau_2& G&\textrm{Action of the generators of $G$ on $E\times F$}\\
         \tline
         1&\textrm{arbitrary}&\Zet_2=\langle \varphi\rangle&
         \varphi {e\choose f}={e+1/2\choose -f}\\
         2&\textrm{arbitrary}&\Zet_2\times\Zet_2=\langle \varphi,\psi\rangle&
         \varphi {e\choose f}={e+1/2\choose -f}, \ \psi {e\choose f}={e+\tau_1/2\choose f+1/2}\\
         3&i&\Zet_4=\langle \varphi\rangle&\varphi {e\choose f}={e+1/4\choose if}\\
         4&i&\Zet_4\times\Zet_2=\langle \varphi,\psi\rangle&
         \varphi {e\choose f}={e+1/4\choose if}, \ \psi {e\choose f}={e+\tau_1/2\choose f+(1+i)/2}\\
         5&\zeta&\Zet_3=\langle \varphi\rangle&
         \varphi {e\choose f}={e+1/3\choose \zeta^2 f}\\
         6&\zeta&\Zet_3\times\Zet_3=\langle \varphi,\psi\rangle&
         \varphi {e\choose f}={e+1/3\choose \zeta^2 f}, \ \psi {e\choose f}={e+\tau_1/3\choose f+(1+\zeta)/3}\\
         7&\zeta&\Zet_6=\langle \varphi\rangle&
         \varphi {e\choose f}={e+1/6\choose \zeta f}
      \end{array}
      $$
      \caption{Action of the generators of $G$ on $E\times F$.}\label{tab:action}
   \end{table}
\end{proposition}

\begin{theorem}[{\cite[Theorem 1.4]{Se1990}}]\label{Thm:Serrano}
   For each of the seven types of bielliptic surfaces,
   a basis of the group $\Num(X)$ of classes of numerically equivalent divisors
   and the multiplicities of the singular fibers in each case are described in Table \ref{tab:multipl&basis}.
   \begin{table}[h]
      $$
      \begin{array}{c|l|l|l}
      \textrm{Type of a bielliptic surface}&G&m_1,\ldots,m_s&\textrm{Basis of $\textrm{Num}(X)$}\\
      \tline
      1&\Zet_2&2,2,2,2&E/2, F\\
      2&\Zet_2\times\Zet_2&2,2,2,2&E/2, F/2\\
      3&\Zet_4&2,4,4&E/4, F\\
      4&\Zet_4\times\Zet_2&2,4,4&E/4, F/2\\
      5&\Zet_3&3,3,3&E/3, F\\
      6&\Zet_3\times\Zet_3&3,3,3&E/3, F/3\\
      7&\Zet_6&2,3,6&E/6, F
      \end{array}
      $$
      \caption{Multiplicities of the singular fibers and a basis of $\textrm{Num}(X)$.}\label{tab:multipl&basis}
   \end{table}
\end{theorem}

   Let $\mu=\textrm{lcm}\{m_1, \ldots, m_s\}$ and let
   $\gamma=|G|$. Note that a  basis of  $\textrm{Num}(X)$
   consists of divisors $E/\mu$ and $\left(\mu/\gamma\right) F$.
   We say that on a bielliptic surface $L$ is a line bundle of
   type $(a,b)$, with respect to the numerical equivalence, or
   $L\equiv (a,b)$ for short,  if $L\equiv a\cdot
   E/\mu+b\cdot(\mu/\gamma) F$. A divisor of type $(0,b)$  with
   $b\in\mathbb{Z}$ is effective if and only if
   $b\cdot\left(\mu/\gamma\right)\in \mathbb{N}$, see
   \cite[Proposition 5.2]{Ap1998}.

\begin{remark}\label{rmk:line-bundles-bielliptic}
   As a result of the previous discussion, we have
   the following properties of line bundles on $X$.
   \begin{itemize}\compact
   \item We have $E^2=0$, $F^2=0$, $E\cdot F=\gamma$, hence if $L_1\equiv (a_1,b_1)$, $L_2\equiv (a_2,b_2)$ then $L_1\cdot L_2=a_1b_2+a_2b_1$.
   \item If $C\equiv(\alpha, \beta)$ is an irreducible curve with $C^2=0$, then $\alpha=0$ or $\beta=0$, and hence $C$ is a fibre (or a multiple of a fibre).
   \end{itemize}
\end{remark}


\subsection*{Acknowledgements}
We warmly thank  Tomasz Szemberg and Krishna Hanumanthu for their helpful remarks.



\footnotesize
\bigskip
   Thomas Bauer,
   Fachbereich Mathematik und Informatik,
   Philipps-Universit\"at Marburg,
   Hans-Meerwein-Stra\ss e,
   D-35032 Marburg, Germany

\nopagebreak
   \textit{E-mail address:} \texttt{tbauer@mathematik.uni-marburg.de}

\bigskip
   {\L}ucja Farnik,
   Department of Mathematics,
   Pedagogical University of Cracow,
   Podchor\c a\.zych 2,
   PL-30-084 Krak\'ow, Poland

\nopagebreak
   \textit{E-mail address:} \texttt{Lucja.Farnik@gmail.com}


\end{document}